\documentclass{article}

\usepackage[utf8]{inputenc}
\usepackage{amsmath}
\usepackage{amsthm}
\usepackage{amssymb}
\usepackage{mathabx}
\usepackage{mathtools}
\usepackage{faktor}
\usepackage{graphicx}

\DeclarePairedDelimiter{\set}{\{}{\}}
\DeclarePairedDelimiter{\abs}{|}{|}

\renewcommand{\phi}{\varphi}

\newcommand{\mc}{\mathcal}

\newcommand{\nat}{\ensuremath{\mathbb{N}}}

\DeclareMathOperator{\Aut}{Aut}

\DeclareMathOperator{\MCG}{MCG}

\newtheorem{theorem}{Theorem}

\newtheorem{lemma}[theorem]{Lemma}
\newtheorem{cor}[theorem]{Corollary}
\newtheorem{fact}[theorem]{Fact}
\newtheorem{remark}[theorem]{Remark}

\newcommand{\loopg}{\mc{L}}
\newcommand{\arcg}{\mc{A}}

\title{Automorphisms of the loop and arc graph of an infinite-type surface}
\author{Anschel Schaffer-Cohen}

\begin{document}

\maketitle

\begin{abstract}
	We show that the extended based mapping class group of an infinite-type surface is naturally isomorphic to the automorphism group of the loop graph of that surface. Additionally, we show that the extended mapping class group stabilizing a finite set of punctures is isomorphic to the arc graph relative to that finite set of punctures. This extends a known result for sufficiently complex finite-type surfaces, and provides a new angle from which to study the mapping class groups of infinite-type surfaces.
\end{abstract}

\section{Introduction and Main Result}\label{intro}

In the tradition of geometric group theory, mapping class groups are often studied via their actions on metric spaces, most famously the Teichmüller space of the underlying surface. The mapping class group of a finite-type orientable surface also acts especially nicely on several simplicial complexes that have been defined for this purpose. These include the curve complex, the pants complex, and the arc complex.

The curve complex and arc complex were defined for general infinite-type surfaces by Aramayona, Fossas, and Parlier in \cite{aramayona}, by analogy with existing definitions in the finite-type case and generalizing a definition for the plane minus a Cantor set given by Bavard in \cite{bavard}. Masur and Minsky \cite{MM} and Masur and Schleimer \cite{MS} showed that both these complexes are infinite-diameter Gromov hyperbolic in the finite case, and Hensel, Przytycki, and Webb \cite{hensel} give an especially nice proof that the hyperbolicity constant does not depend on the surface. In addition, Ivanov \cite{ivanov}, Korkmaz \cite{kork}, Luo \cite{luo}, and Irmak and McCarthy \cite{irmak} showed that the extended mapping class group is isomorphic to the automorphism groups of both complexes.

Hernández Hernández, Morales, and Valdez recently showed in \cite{HMV} that in the infinite-type case the mapping class group is still isomorphic to the automorphism group of the curve graph, which was also shown independently by Bavard, Dowdall, and Rafi in \cite{BDR}. However, the curve graph of an infinite-type surface has diameter $2$, making it quasi-isometrically trivial.

On the other hand, the relative arc graph defined by Aramayona, Fossas, and Parlier is infinite-diameter Gromov hyperbolic, with a hyperbolicity constant that does not depend on the surface \cite{aramayona}. Theorem \ref{main_arc} shows that, with appropriate restrictions, the mapping class group is precisely the group of automorphisms of this graph. We start with Theorem \ref{main}, which shows this result in the special case of the loop graph, an interesting structure in its own right.

For our purposes, a \emph{surface} $\Sigma$ is a connected, oriented $2$-manifold. We call $\Sigma$ \emph{finite-type} if $\pi_1(\Sigma)$ is finitely generated, and \emph{infinite-type} otherwise.

Fix a basepoint $p \in \Sigma$. A \emph{loop} in $\Sigma$ based at $p$ is an unoriented simple closed curve in $\Sigma$ starting and ending at $p$, considered up to isotopy fixing $p$. The \emph{loop graph} $\loopg(\Sigma, p)$ of $\Sigma$ is the graph whose vertex set is the set of all possible loops in $\Sigma$ and where two vertices are connected by an edge if they have representatives intersecting only at $p$.

The \emph{extended based mapping class group} $\MCG^*(\Sigma, p)$ of $\Sigma$ is the group of (possibly orientation-reversing) homeomorphisms $\Sigma \to \Sigma$ that fix the basepoint $p$, considered up to isotopy preserving $p$. By construction, $\MCG^*(\Sigma, p)$ acts on $\loopg(\Sigma, p)$, giving a homomorphism $\MCG^*(\Sigma, p) \to \Aut(\loopg(\Sigma, p))$. The main result of this paper is that this is in fact an isomorphism.

\begin{theorem}\label{main}
	Given an infinite-type surface $\Sigma$ and a basepoint $p \in \Sigma$, the map $\MCG^*(\Sigma, p) \to \Aut(\loopg(\Sigma, p))$ induced by the action is an isomorphism.
\end{theorem}

Irmak and McCarthy \cite{irmak} provide a successful program for proving this theorem in the finite-type case:
\begin{enumerate}
	\item Fix a maximal set of disjoint loops, called a \emph{triangulation}, and observe that its complementary regions are all triangles.
	\item Show that some important local properties---three loops bounding a triangle, two triangles being adjacent, etc.---can be defined in terms of the loop graph and are therefore preserved by automorphisms of that graph. We use some of these results as Facts \ref{IM_triangle} and \ref{IM_adj}.
	\item Use these local properties to construct a homeomorphism inducing a given transformation of our fixed triangulation.
	\item Show that a homeomorphism which fixes a triangulation actually fixes the entire loop graph. We use this result as Fact \ref{IM_enough}.
\end{enumerate}

In our extension to the infinite-type case we have an advantage, a disadvantage, and a trick. The advantage is of course that we can depend on the existing result for finite-type surfaces. The disadvantage is that ``triangulations'' in the infinite-type setting can be much more exotic, as seen in Section \ref{general}. They will in fact have some complementary regions that are not actually triangles. The trick is to notice that Irmak and McCarthy's proof is more general than the result requires: they start by fixing an arbitrary triangulation, but for their proof (and ours) it is sufficient to follow this program with any single triangulation. Thus we can construct a particularly useful triangulation for the specific purpose of building our homeomorphism.

In Section \ref{arc_graph}, we generalize Theorem \ref{main} to the relative arc graph, which we define as follows: let $P$ be a finite set of punctures (i.e.\ isolated planar ends) of an infinite-type surface $\Sigma$. An \emph{arc} in $\Sigma$ relative to $P$ is an unoriented embedded line whose ends approach two (not necessarily distinct) punctures in $\Sigma$, considered up to isotopy.\footnote{An isotopy of an arc cannot change its endpoints, since they are not included in $\Sigma$.} The \emph{relative arc graph} $\arcg(\Sigma, P)$ of $\Sigma$ relative to $P$ is the graph whose vertex set is the set of all possible arcs relative to $P$ in $\Sigma$ and where two vertices are connected by an edge if they have disjoint representatives. The \emph{extended mapping class group} $\MCG^*(\Sigma, P)$ of (possibly orientation-reversing) homeomorphisms stabilizing (but not necessarily fixing) the set $P$ acts on $\arcg(\Sigma, P)$. Then we can prove the following:

\begin{theorem}\label{main_arc}
	Given an infinite-type surface $\Sigma$ and a finite set $P$ of punctures of $\Sigma$, the map $\MCG^*(\Sigma, P) \to \Aut(\arcg(\Sigma, P))$ induced by the action is an isomorphism.
\end{theorem}

There are two special cases to keep in mind: first, if $P$ consists of a single puncture $p$, then $\arcg(\Sigma, P) = \loopg(\Sigma \cup \set{p}, p)$, so Theorem $\ref{main_arc}$ really is a generalization of Theorem $\ref{main}$. Second, if $\Sigma$ has finitely many punctures, then $P$ can contain all of them and so $\MCG^*(\Sigma, P) = \MCG^*(\Sigma)$ and we have a bona fide action of the mapping class group.

The idea of the proof of Theorem \ref{main_arc} will be to reduce it to that of Theorem \ref{main} by picking a puncture $p \in P$ and considering $\loopg(\Sigma \cup \set{p}, p)$ as an induced subgraph of $\arcg(\Sigma, P)$. The main hurdle will therefore be to prove that an automorphism of the arc graph preserves properties like ``this arc is actually a loop'' and ``these two loops are based at the same puncture''.

\subsection*{Acknowledgements}

The author would like to acknowledge the immense help of his adviser, David Futer, as well as fruitful conversations with Priyam Patel and Nicholas G. Vlamis and an insightful email from Edgar Bering.

\section{Triangulations of infinite-type surfaces}\label{general}

Irmak and McCarthy \cite{irmak} define a triangulation as a maximal set of disjoint arcs, or equivalently as a set of arcs whose complementary regions are all triangles. Hatcher \cite{hatcher}, allowing for punctures not in $P$, observes that the complementary regions of a triangulation may also include punctured monogons. In the infinite-type case triangulations are considerably more exotic, as we shall see. So we simply define a \emph{triangulation}\footnote{This somewhat misleading name is (we hope) justified by the fact that it is a straightforward extention of the usage in \cite{irmak} and \cite{hatcher}.} to be a maximal clique in $\arcg(\Sigma, P)$. The following facts are then immediate:

\begin{lemma}
	Given a set $T$ of arcs and an automorphism $f$ of $\arcg(\Sigma, P)$, $T$ is a triangulation if and only if $f(T)$ is a triangulation.
\end{lemma}

\begin{lemma}\label{extend}
	Any set of disjoint arcs can be extended to a triangulation.
\end{lemma}
\begin{proof}
	This follows from Zorn's Lemma.
\end{proof}

A finite-type surface admits countably many triangulations; each triangulation has the same number of loops depending only on the Euler characteristic of the surface; and any two triangulations are connected by a finite sequence of \emph{elementary moves}, in which an arc $\alpha$ is removed from the triangulation and replaced with $\beta$, where the geometric intersection number $\iota(\alpha, \beta) = 1$ \cite{hatcher}. An infinite-type surface, on the other hand, admits uncountably many triangulations by application of Lemma \ref{extend}, each with countably many arcs, and so most pairs of triangulations will not be connected by elementary moves.

The topology of a triangulation is also interesting in the infinite-type setting: consider a finite union $C$ of small circles centered at each $p \in P$ that intersects each arc finitely many times. By compactness, there must be points on $C$ that are the limits of points on distinct arcs of the triangulation. A priori there is no reason to assume these limit points are contained in the triangulation at all, and so the triangulation may not be closed---in fact, we conjecture that it never is.

These and other questions about the nature of triangulations in general will not be studied further in this paper. However, they suggest possible future areas of research, and motivate the construction, in the next section, of a special triangulation to overcome the potential pitfalls of an arbitrary one.

\section{Building a useful triangulation}\label{useful}

For Sections \ref{useful} and \ref{constructing} we fix an infinite-type surface $\Sigma$, a basepoint $p$, and an automorphism $f \colon \loopg(\Sigma, p) \to \loopg(\Sigma, p)$.

Our goal in this section is to build a special triangulation, $\mc{T}$, containing loops that will allow us to break $\Sigma$ down nicely into compact subsurfaces, with $\mc{T}$ restricting to a triangulation on each such subsurface. We will start by building an embedded tree in $\mc{S} \subseteq \Sigma$, called the \emph{skeleton} of $\Sigma$, that has one infinite branch for each end of $\Sigma$.

Fix a pants decomposition $\set{\Pi_0, \Pi_1, \dotsc}$ of $\Sigma$, with $p \in \Pi_0$ and such that $\bigcup_{i=0}^n \Pi_i$ is connected for each $n$. Note that, since $\Sigma$ may have punctures, the holes in a pair of pants may all be boundary components connected to other pairs of pants, or one or two of them may be punctures. For the purpose of constructing the skeleton $\mc{S}$ we distinguish between four types of pairs of pants: namely, for each $\Pi_i$ we can ask how many of its boundary components connect to some $\Pi_j$ for $j < i$. In our inductive definition of $\mc{S}$, we will refer to such a boundary component as \emph{already connected}.

\begin{figure}
	\includegraphics[width=\textwidth]{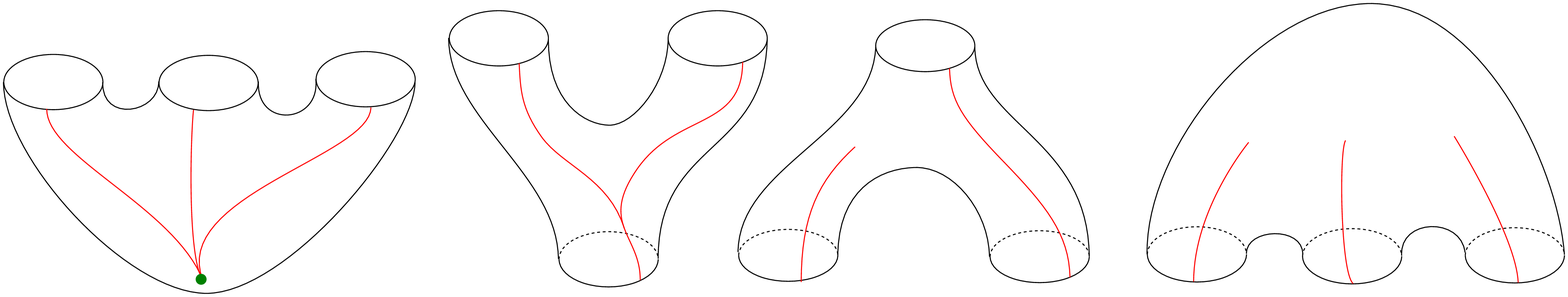}
	\caption{The construction of the skeleton on each type of pair of pants $\Pi_i$. The green dot is the basepoint $p$, and the red curves are pieces of the skeleton. \label{skeleton}}
\end{figure}

In Figure \ref{skeleton} we show how to draw the skeleton on each of the four types of pairs of pants; the first type, with no boundary components already connected, occurs only for $\Pi_0$. Note that by construction $\mc{S}$ is indeed a tree, and its infinite based rays are in bijection with the ends of $\Sigma$; there are also some paths in $\mc{S}$ that end after a finite distance, which do not therefore correspond to ends of $\Sigma$.

\begin{figure}
	\includegraphics[width=\textwidth]{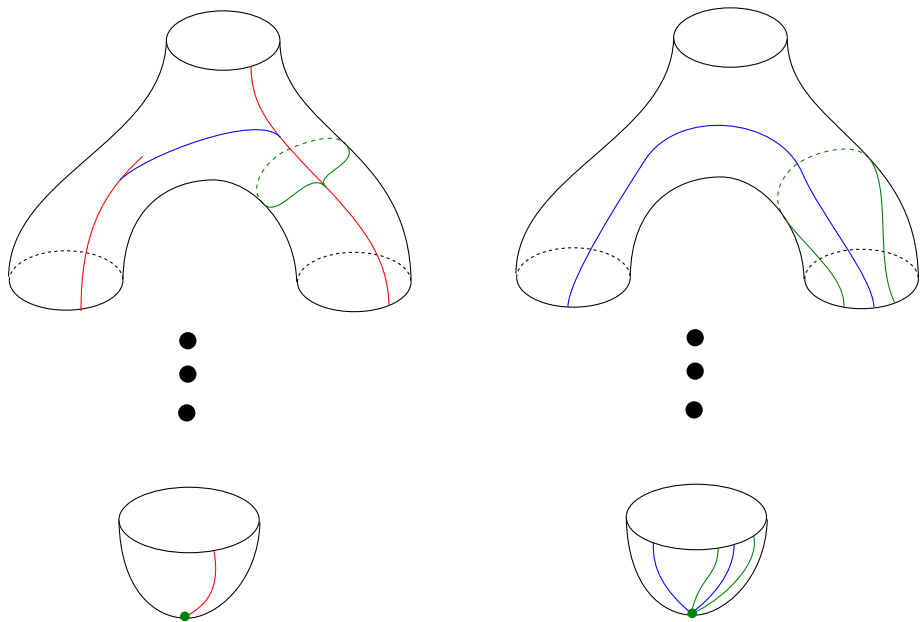}
	\caption{Each arc with endpoints on the skeleton represents a loop based at $p$. \label{arcs_to_loops}}
\end{figure}

The loops of $\mc{T}$ will be defined with reference to the skeleton in the following way: when we draw an arc from one point on the skeleton to another that does not otherwise intersect the skeleton, we are indicating the loop that starts at $p$, takes the unique nonbacktracking path to one point on the skeleton, follows the arc to the other point on the skeleton, and takes the unique nonbacktracking path from that point back to $p$; see Figure \ref{arcs_to_loops} for some examples. Note that two such arcs, if disjoint, indicate disjoint loops.

Although we have defined our loops in terms of arcs with endpoints on $\mc{S}$, this is purely a notational convenience. After this section we will consider only the loops themselves, and ignore the arcs with which they were defined.

\begin{figure}
	\includegraphics[width=\textwidth]{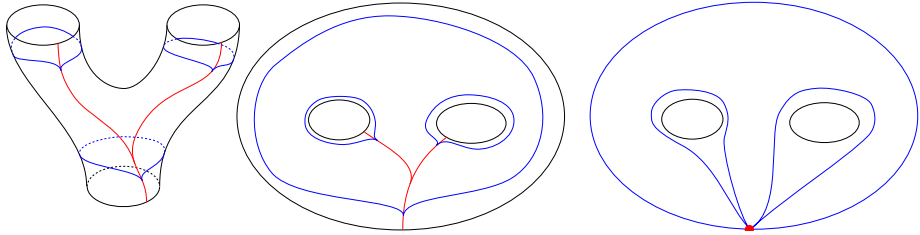}
	\caption{The skeleton and loops on a pair of pants with at most one boundary component already connected. The first and second figures are homeomorphic; in the third figure, the skeleton has been contracted to a point so that the triangle dicomposition is more clearly visible. \label{loops_1}}
\end{figure}

\begin{figure}
	\includegraphics[width=\textwidth]{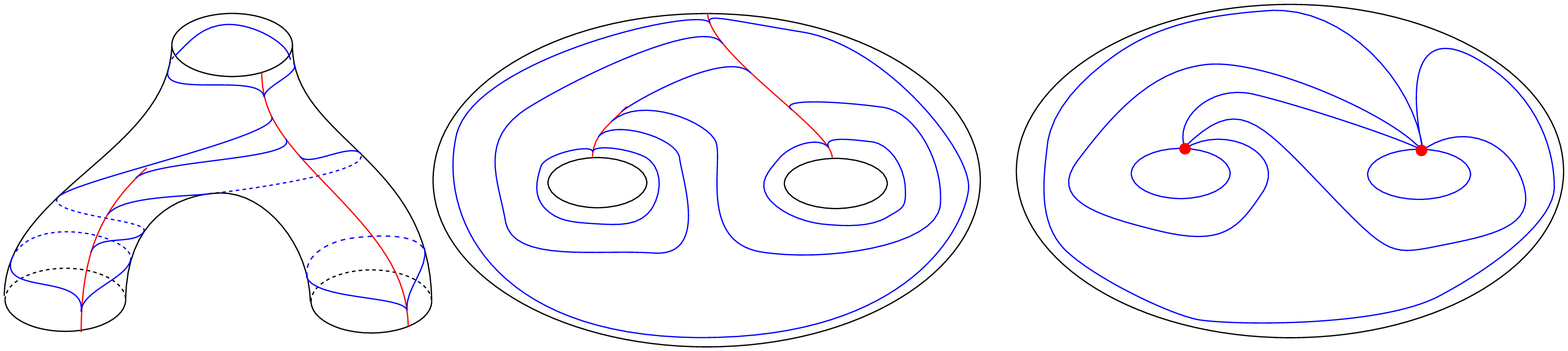}
	\caption{The skeleton and loops on a pair of pants with two boundary components already connected. As in Figure \ref{loops_1}, the first and second figures are homeomorphic; in the third figure, the skeleton has been contracted to two points. \label{loops_2}}
\end{figure}

\begin{figure}
	\includegraphics[width=\textwidth]{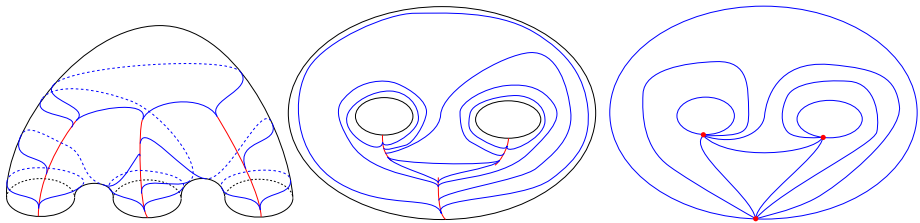}
	\caption{The skeleton and loops on a pair of pants with three boundary components already connected. As in Figure \ref{loops_1}, the first and second figures are homeomorphic; in the third figure, the skeleton has been contracted to three points. \label{loops_3}}
\end{figure}

We now draw the loops of our triangulation on each pair of pants according to the arcs pictured in Figures \ref{loops_1}--\ref{loops_3}. The skeleton in a pair of pants with zero boundary components already connected is identical to the skeleton in a pair of pants with one boundary component already connected, so both of these cases are covered in Figure \ref{loops_1}, with the other two cases covered in Figures \ref{loops_2} and \ref{loops_3} respectively. Note that some of these loops are redundant, as representatives of the same homotopy class may be drawn on more than one pair of pants.

\begin{remark}
	It turns out, somewhat surprisingly, that we will not need to know whether $\mc{T}$ is in fact a triangulation in the sense of Section \ref{general}. But the reader can easily verify that it is in fact a maximal set of loops; more importantly, $\mc{T}$ restricts to a triangulation on each compact subsurface $\bigcup_{i=0}^n \Pi_i$.
\end{remark}

\section{Constructing a homeomorphism}\label{constructing}
We first note a few facts due to Irmak and McCarthy; keep in mind that loops are simply a special case of arcs, and the loop graph a special case of the arc graph.
\begin{fact}[Propositions 3.3 and 3.4 of \cite{irmak}]\label{IM_triangle}
	The condition that three arcs bound a triangle, or that two arcs bound a degenerate triangle,\footnote{A \emph{degenerate triangle} is one where two sides are formed by the same arc.} is preserved under automorphisms of the arc graph.
\end{fact}
\begin{fact}[Propositions 3.5--3.7 of \cite{irmak}]\label{IM_adj}
	When two triangles are adjacent (i.e.\ share one or two edges) their relative orientations are preserved under automorphisms of the arc graph.
\end{fact}
\begin{fact}[Proposition 3.8 of \cite{irmak}]\label{IM_enough}
	If a homeomorphism of a finite-type surface preserves some triangulation up to isotopy, then it induces the identity on the arc graph of that surface.
\end{fact}
The proofs of the first two facts do not depend on the surface being finite-type; they are based exclusively on local properties of the arc graph. Since our surface may have punctures, some loops may bound punctured monogons as well as triangles, and this property will also be preserved:
\begin{lemma}\label{monogon}
	The condition that a loop bounds a punctured monogon is preserved under automorphisms of the arc graph.
\end{lemma}
\begin{proof}
	If a nontrivial loop $\lambda$ does not bound a punctured monogon, then $\Sigma \setminus \lambda$ has either at least two punctures or at least one handle in each component.\footnote{There may be one or two components, depending on whether $\lambda$ is separating.} In either case, we can draw two disjoint triangles adjacent to $\lambda$, which we cannot do if $\lambda$ does bound a punctured monogon. Since the property of three arcs bounding a triangle is preserved by Fact \ref{IM_triangle}, so is the property of bounding a punctured monogon.
\end{proof}

We will first construct an exhaustion of $\Sigma$ by finite-type surfaces $\Sigma_n$, with corresponding homeomoprhisms $\phi_n$ from $\Sigma_n$ to an appropriate subsurface of $\Sigma$. The natural choice is to let the subsurface $\Sigma_n \subseteq \Sigma$ be the union $\bigcup_{i=0}^n \Pi_n$ of the first $n$ pairs of pants. Note that by our choice of pants decomposition, $\Sigma_n$ is connected and contains $p$.

However, this definition is somewhat unhelpful for the purpose of constructing $\phi_n$, because there is no obvious choice for the image $\phi_n(\Sigma_n)$---after all, $f$ will not in general preserve our pants decomposition. But $\Sigma_n$ has a useful alternate definition. If we let $\mc{T}_n$ be the set of loops in $\mc{T}$ supported on $\Sigma_n$, then $\Sigma_n$ is also the largest subsurface of $\Sigma$ (up to isotopy) filled by the loops of $\mc{T}_n$.

We now have a natural choice for the image of $\phi_n$: since $\Sigma_n$ is the largest subsurface filled by $\mc{T}_n$, $\phi_n(\Sigma_n)$ should be the largest subsurface filled by $\set{f(\lambda) \mid \lambda \in \mc{T}_n}$. We can in fact construct such a map:

\begin{lemma}\label{partial_maps}
	For $n \in \nat$, there exists a map $\phi_n : \Sigma_n \to \Sigma$, a homeomorphism onto its image, such that $\phi_n(\lambda) = f(\lambda)$ for each $\lambda \in \loopg(\Sigma, p)$ supported on $\Sigma_n$. In addition, these homeomorphisms are compatible: that is, after applying some isotopy to $\phi_{n+1}$, we have $\left. \phi_{n+1}\right|_{\Sigma_n} = \phi_n$.
\end{lemma}
\begin{proof}
	The image of $\phi_n$ will be the largest subsurface filled by $f(\mc{T}_n) = \set{f(\lambda) \mid \lambda \in \mc{T}_n}$; call this subsurface $\Omega_n$. Whenever three loops $\lambda_1, \lambda_2, \lambda_3 \in \mc{T}_n$ bound a triangle in $\Sigma$, so do their images $f(\lambda_1), f(\lambda_2), f(\lambda_3)$ by Fact \ref{IM_triangle}. The latter triangle is thus included in $\Omega_n$ and can be the image of the former. When two such triangles are adjacent, their orientations are preserved by Fact \ref{IM_adj} and so the homeomorphisms on adjacent triangles can be stitched together after applying some isotopy.

	When a loop $\lambda \in \mc{T}_n$ bounds a punctured monogon in $\Sigma$, so does $f(\lambda)$ by Lemma \ref{monogon}. Then by construction the punctured monogon bounded by $f(\lambda)$ is contained in $\Omega_n$, and the two punctured monogons are homeomorphic. This homeomorphism can be stitched to those above along $\lambda$ and $f(\lambda)$.

	When a loop $\lambda \in \mc{T}_n$ is parallel to the boundary of $\Sigma_n$, it must bound a triangle in $\mc{T}$ with at least one side not in $\mc{T}_n$. By Fact \ref{IM_triangle}, $f(\lambda)$ therefore bounds a triangle in $f(\mc{T})$ with at least one side not in $f(\mc{T}_n)$ and so this triangle is not included in $\Omega_n$. Thus $f(\lambda)$ is parallel to the boundary of $\Omega_n$, and we can extend our homeomorphism to a tubular neighborhood of $\lambda$.
	
	Since $\Sigma_n$ is made up of tubular neighborhoods of loops in $\mc{T}_n$, triangles bounded by loops in $\mc{T}_n$, and punctured monogons bounded by loops in $\mc{T}_n$---and likewise for $\Omega_n$ and loops in $f(\mc{T}_n)$---this algorithm gives a homeomorphism $\phi_n : \Sigma_n \to \Omega_n$. Since $\phi_n(\lambda) = f(\lambda)$ for each $\lambda \in \mc{T}_n$, which is a triangulation of $\Sigma_n$, it follows by Fact \ref{IM_enough} that $\phi_n(\lambda) = f(\lambda)$ for every $\lambda \in \loopg(\Sigma, p)$ supported on $\Sigma_n$.
	
	By construction $\Omega_n \subseteq \Omega_{n+1}$, and the only choices we made in defining our homeomorphisms were isotopies on the interiors of triangles and punctured monogons. Thus after an isotopy, $\phi_n$ and $\phi_{n+1}$ agree on $\Sigma_n$.
\end{proof}

To prove Theorem \ref{main} we need to combine these partial maps $\phi_n$, and we also need to prove injectivity. The following lemma achieves the latter result.

\begin{lemma}\label{homeo_trivial}
	If a homeomorphism $\phi : \Sigma \to \Sigma$ induces the identity automorphism of $\Aut(\loopg(\Sigma, p))$ then $\phi$ is isotopic to the identity.
\end{lemma}
\begin{proof}
	The key insight here is that if $\phi$ preserves the isotopy class of each based loop it must also preserve the isotopy class of each free loop. So if we fix a pants decomposition of $\Sigma$, $\phi$ will preserve its boundary curves up to isotopy. Then by Corollary 1.2 of \cite{alexander}, $\phi$ is isotopic to the identity.
\end{proof}

\begin{cor}\label{bijective}
	The homomorphism $\MCG^*(\Sigma, p) \to \Aut(\loopg(\Sigma, p))$ is bijective.
\end{cor}
\begin{proof}
	Injectivity follows directly from Lemma \ref{homeo_trivial}.

	To prove surjectivity, fix $f \in \Aut(\loopg(\Sigma, p))$, construct $\set{\phi_n}_{n \in \nat}$ as in Lemma \ref{partial_maps}, and apply isotopies so that each $\phi_n$ is a restriction of $\phi_{n+1}$. Define $\phi : \Sigma \to \Sigma$ by letting $\phi(x) = \phi_n(x)$ for some $n$ where $x \in \Sigma_n$. Since the $\Sigma_n$ exhaust $\Sigma$ and $\phi_n$ agrees with $\phi_m$ wherever both are defined, this map $\phi$ is well-defined. Since it is a homeomorphism on each $\Sigma_n$ and the $\Sigma_n$ exhaust $\Sigma$, it is a homeomorphism onto its image. And since the $\Omega_n$ also exhaust $\Sigma$, this image is in fact $\Sigma$, so $\phi : \Sigma \to \Sigma$ is a homeomorphism. Thus $[\phi] \in \MCG^*(\Sigma, p)$, and its image is $f$.
\end{proof}

\section{The relative arc graph}\label{arc_graph}
The contents of Theorem \ref{main} can be generalized to the case of the \emph{relative arc graph} $\arcg(\Sigma, P)$. Recall the definition of the relative arc graph from Section \ref{intro}, where we noted that that the loop graph $\loopg(\Sigma, p)$ is simply a special name for the relative arc graph $\arcg(\Sigma \setminus \set{p}, \set{p})$; more generally, if $p \in P$ then $\loopg(\Sigma \cup \set{p}, p)$ is an induced subgraph of $\arcg(\Sigma, P)$. So for the remainder of this section we pick some $p \in P$ and let $\Sigma' = \Sigma \cup \set{p}$. This gives an immediate result:

\begin{lemma}
	The homomorphism $\MCG^*(\Sigma, P) \to \Aut(\arcg(\Sigma, P))$ is injective.
\end{lemma}
\begin{proof}
	A homeomorphism $\phi : \Sigma \to \Sigma$ that induces the identity automorphism on $\arcg(\Sigma, P)$ must also induce the identity automorphism on $\loopg(\Sigma', p)$ and so by Lemma \ref{homeo_trivial} it is isotopic to the identity.
\end{proof}

To show surjectivity, let $f$ be an arbitrary automorphism of $\arcg(\Sigma, P)$. In Section \ref{useful}, we built a special triangualation $\mc{T}$. In this section, we will instead start with an arbitrary triangulation $T'$ in $\loopg(\Sigma', p)$ as described in Section \ref{general} and extend it to a triangulation $T$ in $\arcg(\Sigma, P)$.

\begin{lemma}\label{extend_to_arcg}
	If $T'$ is a triangulation in $\loopg(\Sigma', p)$, then there is a unique extension of $T'$ to a triangulation $T$ in $\arcg(\Sigma, P)$.
\end{lemma}

\begin{figure}
	\centering
	\includegraphics[width=\textwidth/3]{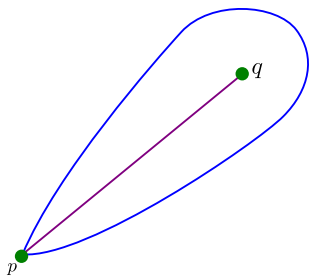}
	\caption{Adding a single arc (shown here in purple) turns a punctured monogon containing $q$ into a degenerate triangle.\label{degen}}
\end{figure}

\begin{proof}
	Consider a puncture $q \in P \setminus \set{p}$. Observe that $T'$ must include a loop that bounds a punctured monogon around $q$---if not, such a loop could be added to $T'$, contradicting maximality. The only way to extend $T'$ with an arc having an endpoint at $q$ is to turn this monogon into a degenerate triangle as in Figure \ref{degen}. Repeating this process for each $q \in P \setminus \set{p}$ will give the only triangulation in $\arcg(\Sigma, P)$ containing $T'$. Call this triangulation $T$.
\end{proof}

Note that each arc of $T$ is either contained in $T'$---in which case it is a loop based at $p$---or is one of these new arcs, each of which forms two sides of a single degenerate triangle in $T'$. The following lemma shows that this condition holds after applying the automorphism $f$.

\begin{lemma}\label{phi_p_T}
	Let $T'$ be a triangulation in $\loopg(\Sigma', p)$ and $T$ its extension to $\arcg(\Sigma, P)$ by Lemma \ref{extend_to_arcg}. Then there is a puncture $\phi(p) \in P$ such that for every loop $\lambda \in T'$ based at $p$, $f(\lambda)$ is a loop based at $\phi(p)$, and for every arc $\alpha \in T$ with exactly one endpoint at $p$, $f(\alpha)$ has exactly one endpoint at $\phi(p)$.
\end{lemma}

\begin{figure}
	\centering
	\includegraphics[width=\textwidth/3]{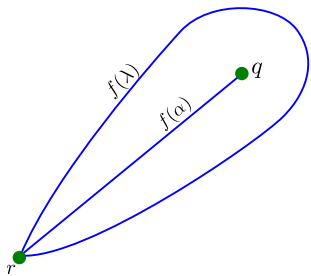}
	\caption{The image of a degenerate triangle under $f$\label{degen_labeled}}
\end{figure}

\begin{proof}
	Let $n = \abs{P}$. By construction $T$ has exactly $n-1$ degenerate triangles, and Fact \ref{IM_triangle} ensures that $f(T)$ has the same number of degenerate triangles. Let us label one such degenerate triangle as in Figure \ref{degen_labeled}. Observe that any arc in $\arcg(\Sigma, P)$ with an endpoint at $q$ either is $f(\alpha)$ or intersects $f(\lambda)$. Thus the only arc in $f(T)$ with an endpoint at $q$ is $f(\alpha)$.

	By applying the above argument to each degenerate triangle of $f(T)$, we find $n-1$ punctures, each of which is the endpoint of exactly one arc, leaving one puncture left over. This puncture will be $\phi(p)$, and note that in Figure \ref{degen_labeled} the puncture $r$ is the endpoint of at least two arcs ($f(\lambda)$ and $f(\alpha)$) and so $r = \phi(p)$.

	If $\alpha$ is an arc in $T$ with exactly one endpoint at $p$, then by construction it is the doubled edge in a degenerate triangle and so $f(\alpha)$ has exactly one endpoint at $\phi(p)$. All other arcs in $f(T)$ must have both endpoints at $\phi(p)$ because there are no other punctures in $P$ available. It follows that all other arcs in $f(T)$---that is, $f(\lambda)$ for every $\lambda \in T'$---are loops based at $\phi(p)$.
\end{proof}

We would like $\phi(p)$ to be defined independently of our choice of a triangulation $T'$ in the statement of Lemma \ref{phi_p_T}, but this identification is not immediately obvious.

\begin{lemma}\label{phi_p}
	For any loop $\lambda \in \loopg(\Sigma', p)$, $f(\lambda)$ is a loop based at $\phi(p)$.
\end{lemma}

\begin{proof}
	By extending $\lambda$ to a triangulation in $\loopg(\Sigma', p)$ via Lemma \ref{extend}, we see by Lemma \ref{phi_p_T} that $f(\lambda)$ is indeed a loop. If $\mu \in \loopg(\Sigma', p)$ is disjoint from $\lambda$, then the set $\set{\lambda, \mu}$ can also be extended to a triangulation in $\loopg(\Sigma', p)$, which means $f(\lambda)$ and $f(\mu)$ are based at the same puncture.

	Even if $\lambda$ and $\mu$ are not disjoint, the loop graph is connected (Theorem 1.1 of \cite{aramayona}) and so there is a path $\lambda = \lambda_0, \dotsc, \lambda_k = \mu$ so that each $\lambda_i$ and $\lambda_{i+1}$ are disjoint. Then $f(\lambda_i)$ and $f(\lambda_{i+1})$ are based at the same puncture, and so by induction are $f(\lambda)$ and $f(\mu)$.

	Thus the image of every loop will be based at $\phi(p)$, regardless of which triangulation was used to find $\phi(p)$ in Lemma \ref{phi_p_T}.
\end{proof}

We now have very nearly all the ingredients necessary to apply the results of Section \ref{constructing}. If $\phi(p) = p$, then Lemma \ref{phi_p} means $f \in \Aut(\loopg(\Sigma', p))$ and we can apply Corollary \ref{bijective} directly. If not, we need only a bit more bookkeeping.

\begin{lemma}
	There exists a homeomorphism $\phi \in \MCG^*(\Sigma, P)$ inducing the automorphism $f$ on $\arcg(\Sigma, P)$.
\end{lemma}

\begin{proof}
	Let $\psi$ be a homeomorphism of $\Sigma$ that transposes $p$ and $\phi(p)$ while otherwise fixing the ends of $\Sigma$, and let $g$ be the automorphism of $\arcg(\Sigma, P)$ induced by $\psi$. Then consider $f' = g \circ f$. By Lemma \ref{phi_p} and the construction of $g$, every loop $\lambda \in \loopg(\Sigma', p)$ is mapped to a loop $f'(\lambda)$ based at $p$. In other words, $f' \in \Aut(\loopg(\Sigma', p))$, and so by Corollary \ref{bijective} it is induced by a homeomorphism $\phi' \in \MCG^*(\Sigma', p) \subseteq \MCG^*(\Sigma, P)$. Let $\phi = \psi^{-1} \circ \phi'$.

	By construction, $\phi$ agrees with $f$ on every loop based at $p$; in particular they agree on the image of $\mc{T}'$ and thus of $\mc{T}$. Then as in the proof of Lemma \ref{partial_maps}, $\phi$ agrees with $f$ on each finite-type subsurface $\Sigma_n$ because it preserves a triangulation of that subsurface. Since every arc is contained in $\Sigma_n$ for high enough $n$, $\phi$ must agree with $f$ on all arcs, and so $\phi$ induces $f$.
\end{proof}

\bibliographystyle{alpha}
\bibliography{refs}{}

\end{document}